\documentclass[a4paper,12pt, reqno]{amsart}
\usepackage{amsmath}
\usepackage{amsfonts}
\usepackage{amssymb}
\usepackage{amsthm}
\usepackage[all]{xy}
\usepackage{txfonts}
\usepackage{textcomp}
 \usepackage{color}
\usepackage{soul}

\newtheorem{theorem}{Theorem}
\newtheorem{lemma}{Lemma}

\newtheorem{proposition}{Proposition}

\newtheorem{remark}{Remark}
\newtheorem{definition}{Definition}

\newcommand{\C}{\mathbb{C}}

\newcommand{\Z}{\mathbb{Z}}

\newcommand{\trace}{{\rm trace}}
\newcommand{\meas}{{\rm meas}}
\newcommand{\tr}{{\rm tr}}
\newcommand{\Ad}{{\rm Ad}}
\newcommand{\Wh}{{\rm Wh}}
\newcommand{\Hom}{{\rm Hom}}
\newcommand{\Int}{{\rm Int}}
\newcommand{\val}{{\rm val}}
\newcommand{\supp}{\rm{supp}}
\newcommand{\g}{\mathfrak{g}}
\newcommand{\bg}{\boldsymbol{\mathfrak{g}}}

\makeatletter
\def\oversortoftilde#1{\mathop{\vbox{\m@th\ialign{##\crcr\noalign{\kern3\p@}%
      \sortoftildefill\crcr\noalign{\kern3\p@\nointerlineskip}%
      $\hfil\displaystyle{#1}\hfil$\crcr}}}\limits}

\def\sortoftildefill{$\m@th \setbox\z@\hbox{$\braceld$}%
  \braceld\leaders\vrule \@height\ht\z@ \@depth\z@\hfill\braceru$}

\makeatother

\title{A theorem of M\oe glin-Waldspurger for covering groups}
\author{Shiv Prakash Patel}
\address{School of Mathematics \\
Tata Institute of Fundamental Research \\
Homi Bhabha Road, Colaba \\
Mumbai - 400005\\
INDIA}
\email{shiv@math.tifr.res.in}


\date{\today}

\begin{document}

\begin{abstract}
Let $E$ be a non-Archimedian local field of characteristic zero and residual characteristic $p$. Let ${\bf G}$ be a connected reductive group defined over $E$ and $\pi$ an irreducible admissible representation of $G={\bf G}(E)$. A result of C. M{\oe}glin and J.-L. Waldspurger (for $p \neq 2$) and S. Varma (for $p=2$) states that the leading coefficient in the character expansion of $\pi$ at the identity element of ${\bf G}(E)$ gives the dimension of a certain space of degenerate Whittaker forms. In this paper we generalize this result of M{\oe}glin-Waldspurger to the setting of covering groups $\tilde{G}$ of $G$.
\end{abstract}
\maketitle

\section{Introduction}  
Let $E$ be a non-Archimedian local field of characteristic zero and ${\bf G}$ a connected split reductive group defined over $E$ and $G= {\bf G}(E)$.  Let $\bg = \text{Lie}({\bf G})$ be the Lie algebra of ${\bf G}$ and $\g = \bg (E)$. Let $(\pi, W)$ be an irreducible admissible representation of $G$. A theorem of F. Rodier, in \cite{Rod75}, relates the dimension of the space of non-degenerate Whittaker forms of $\pi$ and coefficients in the character expansion of $\pi$ around identity. More precisely, Rodier proves that if the residual characteristic of $E$ is large enough and the group ${\bf G}$ is split then the dimension of any space of non-degenerate Whittaker functionals for $(\pi, W)$ equals the coefficient in the character expansion of $\pi$ at identity corresponding to an appropriate maximal nilpotent orbit in the Lie algebra $\g$. Rodier proved his theorem assuming that the residual characteristic of $E$ is large enough, in fact, greater than a constant which depends only on the root datum of ${\bf G}$. A theorem of C. M{\oe}glin and J.-L. Waldspurger \cite{MW87} generalizes this theorem of Rodier in several directions, in particular proving the theorem of Rodier for the fields $E$ whose residual characteristic is odd and removing the assumption on ${\bf G}$ being split. The theorem of M{\oe}glin-Waldspurger is a more precise statement about the coefficients appearing in the character expansion around identity and certain spaces of `degenerate' Whittaker forms. In a recent work of S. Varma \cite{San14} this theorem has been proved for fields with even residual characteristic by modifying certain constructions in \cite{MW87} to accommodate the case of even residual characteristic (see the remark at the end of the introduction). So the theorem of M{\oe}glin-Waldspurger is true for all connected reductive groups without any restriction on the residual characteristic of the field $E$. We now recall the theorem of M{\oe}glin-Waldspurger. To state the theorem we need to introduce some notation. Let $Y$ be a nilpotent element in $\bg$ and suppose $\varphi : \mathbb{G}_{m} \longrightarrow {\bf G}$ is a one parameter subgroup satisfying
\begin{equation} \label{condition:a}
 Ad(\varphi(t))Y=t^{-2}Y.
\end{equation}
Associated to such a pair $(Y,\varphi)$ one can define a certain space $\mathcal{W}_{(Y, \varphi)}$, called the space of degenerate Whittaker forms of $(\pi, W)$ relative to $(Y, \varphi)$ (see Section \ref{degenerate_W_forms}  for the definition).  \\ 

Define $\mathcal{N}_{Wh}(\pi)$ to be the set of nilpotent orbits $\mathcal{O}$ of $\mathfrak{g}$ for which there exists an element $Y \in \mathcal{O}$ and a $\varphi$ satisfying (\ref{condition:a}) such that the space $\mathcal{W}_{(Y, \varphi)}$ of degenerate Whittaker forms  relative to the pair $(Y, \varphi)$ is non-zero. \\ 

Recall that the character expansion of $(\pi, W)$ around identity is a sum $\sum_{\mathcal{O}} c_{\mathcal{O}} \widehat{\mu_{\mathcal{O}}}$, where $\mathcal{O}$ varies over the set of nilpotent orbits of $\mathfrak{g}$, $c_{\mathcal{O}} \in \C$ and $\widehat{\mu_{\mathcal{O}}}$ is the Fourier transform of a suitably chosen measure $\mu_{\mathcal{O}}$ on $\mathcal{O}$. One defines $\mathcal{N}_{tr}(\pi)$ to be the set of nilpotent orbits $\mathcal{O}$ of $\mathfrak{g}$ such that the corresponding coefficient $c_{\mathcal{O}}$ in the character expansion of $\pi$ around identity is non zero. \\ 

We have the standard partial order on the set of nilpotent orbits in $\mathfrak{g}$: $\mathcal{O}_{1} \leq \mathcal{O}_{2}$ if $\mathcal{O}_{1} \subset \overline{\mathcal{O}_{2}}$. Let ${\rm Max}(\mathcal{N}_{Wh}(\pi))$ and ${\rm Max}(\mathcal{N}_{tr}(\pi))$  denote the set of maximal element in $\mathcal{N}_{Wh}(\pi)$ and $\mathcal{N}_{tr}(\pi)$ respectively with respect to this partial order. Then the main theorem of M{\oe}glin-Waldspurger in Chapter I of \cite{MW87} is as follows:
\begin{theorem} \label{theorem:M-W}
Let ${\bf G}$ be a connected reductive group defined over $E$. Let $\pi$ be an irreducible admissible representation of $G={\bf G}(E)$ then
\[
 {\rm Max}(\mathcal{N}_{Wh}(\pi)) =  {\rm Max}(\mathcal{N}_{tr}(\pi)).
\]
Moreover, if $\mathcal{O}$ is an element in either of these sets, then for any $(Y, \varphi)$ as above with $Y \in \mathcal{O}$ we have
\[
c_{\mathcal{O}} = \dim \mathcal{W}_{(Y, \varphi)}.
\]
\end{theorem}

If one considers the case of the pair $(Y, \varphi)$ with $Y$ a `regular' nilpotent element then the above theorem of M{\oe}glin-Waldspurger specializes to Rodier's theorem. \\

In this paper we generalize the theorem of M{\oe}glin-Waldspurger to the setting of a covering group $\tilde{G}$ of $G$. Let $\mu_{r}$ be the group of $r$-th roots of unity in $\C^{\times}$. An $r$-fold covering group $\tilde{G}$ of $G$ is a central extension of locally compact groups by $\mu_{r} := \{ z \in \C \mid z^{r}=1 \}$ giving rise to the following short exact sequence
\begin{equation} \label{def:cover}
 {1} \longrightarrow \mu_r \longrightarrow \tilde{G} \longrightarrow G \longrightarrow {1}.
\end{equation}

The representations of $\tilde{G}$ on which $\mu_{r}$ acts by the natural embedding $\mu_{r} \hookrightarrow \C^{\times}$ are called genuine representations. The definition of the space of degenerate Whittaker forms of a representation of $G$ involves only unipotent groups. Since the covering $\tilde{G} \longrightarrow G$ splits over any unipotent subgroup of $G$ in a unique way, see \cite{MW95}, this makes it possible to define the space of degenerate Whittaker forms for any genuine smooth representation $(\pi, W)$ of $\tilde{G}$. In particular, it makes sense to talk of the set $\mathcal{N}_{Wh}(\pi)$.\\
 
The existence of character expansion of an admissible genuine representation $(\pi, W)$ of $\tilde{G}$ has been proved by Wen-Wei Li in \cite{WWLi}. At identity, the Harish-Chandra-Howe character expansion of an irreducible genuine representation has the same form and therefore we have $\mathcal{N}_{\tr}(\pi)$. This makes it possible to have an analogue of Theorem \ref{theorem:M-W} to the setting of covering groups. The main aim of this paper is to prove the following.
\begin{theorem} \label{main:theorem}
Let $\pi$ be an irreducible admissible genuine representation of $\tilde{G}$. Then
\[
 {\rm Max}(\mathcal{N}_{Wh}(\pi)) =  {\rm Max}(\mathcal{N}_{tr}(\pi)).
\]
Moreover, if $\mathcal{O}$ is an element in either of these sets, then for any $(Y, \varphi)$ as above with $Y \in \mathcal{O}$ we have
\[
c_{\mathcal{O}} = \dim \mathcal{W}_{(Y, \varphi)}.
\]
\end{theorem}
 We will use the work of M{\oe}glin-Waldspurger \cite{MW87} and to accommodate the even residual characteristic case, we follow Varma \cite{San14}. Let us describe some ideas involved in the proof. Let $Y$ be a nilpotent element in $\g$ and $\varphi$ a one parameter subgroup as above. Let $\bg_{i}$ be the eigenspace of weight $i$ under the action of $\mathbb{G}_{m}$ on $\bg$ via $\Ad \circ \varphi$. One can attach a parabolic subgroup ${\bf P}$ with unipotent radical ${\bf N}$ whose Lie algebras are $\boldsymbol{\mathfrak{p}} := \oplus_{i \geq 0} \bg_{i}$ and  $\boldsymbol{\mathfrak{n}}= \oplus_{i > 0} \bg_{i}$ respectively. The one parameter subgroup $\varphi $ also determines a parabolic subgroup ${\bf P}^{-}$ opposite to ${\bf P}$ with Lie algebra $\boldsymbol{\mathfrak{p}}^{-} = \oplus_{i \leq 0} \bg_{i}$. For simplicity, assume $\bg_{1} =0$ for the purpose of the introduction. Then $\boldsymbol{\mathfrak{n}} = \oplus_{i \geq 2} \bg_{i}$ and $\chi : \gamma \mapsto \psi(B(Y, \log \gamma))$ defines a character of $N= {\bf N}(E)$, where $B$ is an $\Ad(G)$-invariant non-degenerate symmetric bilinear form on $\g$ and $\psi$ is an additive character of $E$. In this case (i.e., $\bg_{1}=0$), the space of degenerate Whittaker forms $\mathcal{W}_{(Y,\varphi)}$ is defined to be the twisted Jacquet module of $\pi$ with respect to $(N, \chi)$. In the case where $\bg_{1} \neq 0$, the definition of $\mathcal{W}_{(Y, \varphi)}$ needs to be appropriately modified (see Section \ref{degenerate_W_forms}).\\
 
 On the other hand, to the pair $(Y, \varphi)$ one attaches certain open compact subgroups $G_{n}$ of $G$ for large $n$ and certain characters $\chi_{n}$ of $G_{n}$. One then proves that the covering $\tilde{G} \longrightarrow G$ splits over $G_{n}$ for large $n$, so that $G_{n}$ can be seen as subgroups of $\tilde{G}$ as well. Let $t := \varphi(\varpi)$ and $\tilde{t}$ be any lift of $t$ in $\tilde{G}$. It turns out that ${\tilde{t}}^{-n} G_{n} {\tilde{t}}^{n} \cap N$ becomes ``arbitrarily large'' subgroup of $N$ and ${\tilde{t}}^{-n}G_{n}{\tilde{t}}^{n} \cap P^{-}$ gets ``arbitrarily small'' subgroup of $P^{-}$, as $n$ becomes large. For large $n$, the characters $\chi_{n}$ have been so defined that that the character $\chi_{n}' := \chi_{n} \circ \Int({\tilde{t}}^{n})$ restricted to ${\tilde{t}}^{-n} G_{n} {\tilde{t}}^{n} \cap N$ agrees with $\chi$. Using Harish-Chandra-Howe character expansion one proves that the dimension of $(G_{n}, \chi_{n})$-isotypic component of $W$ is equal to $c_{\mathcal{O}}$ for large enough $n$, where $\mathcal{O}$ is the nilpotent orbit of $Y$ in $\g$. 
 Finally one proves that there is a natural isomorphism between $(\tilde{t}^{-n}G_{n}\tilde{t}^{n}, \chi_{n} \circ \Int(\tilde{t}^{n}))$-isotypic component of $W$ and $\mathcal{W}_{(Y, \varphi)}$.
\begin{remark}
 The definition of $\mathcal{W}_{(Y, \varphi)}$ (hence that of $\mathcal{N}_{\Wh}(\pi)$) depends on a choice of an additive character $\psi$  of $E$ and a choice of $\Ad(G)$-invariant non-degenerate bilinear form $B$ on $\g$. On the other hand, in the character expansion, $c_{\mathcal{O}}$'s (hence $\mathcal{N}_{\tr}(\pi)$) depend on $\psi$, $B$, a measure on $G$ and a measure on $\g$. However by choosing a compatible measure on $G$ and $\g$ via $\exp$ map one gets rid of the dependency of $c_{\mathcal{O}}$ on these measures on $G$ and $\g$ and therefore depends only on $\psi$ and $B$. For more detailed discussion about the dependency on $B$ and $\psi$ on the results here, see Remark 4 in \cite{San14}.
 \end{remark} 
\begin{remark} 
One aspect in Verma's proof for $p=2$, which does not obviously generalise from the proof of $p \neq 2$ is the prescription of the character $\chi_{n}$ of $G_{n}$ given in \cite{MW87}, which is due to somewhat bad behaviour of Campbell-Hausdorff formula in $p=2$ case. Using Kirillov theory of compact $p$-adic groups Varma prescribes a $\chi_{n}$ (although not unique) which will serve our purpose. On the other hand, the definition of degenerate Whittaker forms of $W$ has also been modified by Varma to accommodate the case $p=2$.
\end{remark} 
Although the methods used in the paper are not new and heavily depend on the proofs in the linear case, the result is useful in the study of representation theory of covering groups. Author himself has made use of this result in his thesis, where he attempts to generalize a result of D. Prasad in \cite{Prasad92} in the setting of covering groups, namely, in the harmonic analysis relating the pairs $(\widetilde{{\rm GL}_{2}(E)}, {\rm GL}_{2}(F))$ and $(\widetilde{{\rm GL}_{2}(E)}, D_{F}^{\times})$, where $E/F$ is a quadratic extension of non-Archimedian local field, $D_{F}$ is the quaternion division algebra with center $F$ and $\widetilde{{\rm GL}_{2}(E)}$ is a certain two fold cover of ${\rm GL}_{2}(E)$.

Let us briefly give an outline of the organization of the paper. In Section \ref{G_n ans chi_n}, we recall the definition of the subgroups $G_{n}$ and state some properties of the character $\chi_{n}$. In Section \ref{covering_groups}, we recall splitting of the covering groups over $G_{n}$ and describe an appropriate choice of the splitting over the subgroup $G_{n}$ for large $n$. In Section \ref{degenerate_W_forms} we give the definition of the space of degenerate Whittaker forms and describe important set up to prove the main theorem. In Section \ref{proof of main theorem}, we transfer some results from linear groups to covering groups in a few lemmas and based on these lemmas we prove the main theorem. \\

{\bf Acknowledgements:}  Author would like to express his gratitude to Professor D. Prasad and Professor Sandeep Varma for their numerous help and suggestions at various points. Without their help and continuous encouragement this paper would not have been possible.

\section{Subgroups $G_{n}$ and characters $\chi_{n}$} \label{G_n ans chi_n}
In this section, we recall a certain sequence of subgroups $G_{n}$ of $G$, which form a basis of neighbourhoods at identity and certain characters $\chi_{n} : G_{n} \longrightarrow \C^{\times}$. Although the objects involved in this section were defined for linear groups in \cite{MW87, San14}, we will lift them to covering groups in a suitable way in Section \ref{covering_groups} and work with these lifts in this paper. \\

Let $\mathfrak{O}_{E}$ denotes the ring of integers in $E$. We fix an additive character $\psi$ of $E$ with conductor $\mathfrak{O}_{E}$. Fix an $\Ad(G)$-invariant non-degenerate symmetric bilinear form $B : \g \times \g \longrightarrow E$. \\

Let $Y$ be a nilpotent element in $\mathfrak{g}$. Choose a one parameter subgroup $\varphi : \mathbb{G}_{m} \longrightarrow {\bf G}$ satisfying
\begin{equation} \label{property:1}
\Ad(\varphi(s))Y = s^{-2}Y, \forall s \in \mathbb{G}_{m}.
\end{equation}
Existence of such a $\varphi$ is known from the theory of $\mathfrak{sl}_{2}$-triplets. But there are examples which do not come from this theory. We note that for a given nilpotent element $Y \in \g$ the existence of $\varphi$ is guaranteed by the theory of $\mathfrak{sl}_{2}$-triplets.\\
For $i \in \Z$, define
\[
\bg_{i} = \{ X \in \bg : \Ad(\varphi(s))X = s^{i}X, \forall s \in \mathbb{G}_{m} \}.
\]
Set 
\[
\boldsymbol{\mathfrak{n}}:=\boldsymbol{\mathfrak{n}}^{+}:= \oplus_{i >0} \bg_{i}, \boldsymbol{\mathfrak{n}}^{-}:= \oplus_{i<0} \bg_{i}, \boldsymbol{\mathfrak{p}}^{-}:= \oplus_{i \leq 0} \bg_{i}.
\]
The parabolic subgroup ${\bf P}^{-}$ of ${\bf G}$ stabilizing $\boldsymbol{\mathfrak{n}}^{-}$ has $\boldsymbol{\mathfrak{p}}^{-}$ as its Lie algebra. Let ${\bf N}= {\bf N}^{+}$ be the unipotent subgroup of ${\bf G}$ having the Lie algebra $\boldsymbol{\mathfrak{n}}$.\\

Let $G(Y)$ be the centralizer of $Y$ in $G$ and $Y^{\#}$  the centralizer of $Y$ in $\g$. The $G$-orbit $\mathcal{O}_{Y}$ of $Y$ can be identified with $G/G(Y)$ and therefore its tangent space at $Y$ can be identified with $\g/Y^{\#}$. Note that 
\[
\begin{array}{lcl}
Y^{\#} &=& \{ X \in \g : [X, Y]=0 \} \\
           &=& \{ X \in \g : B([X,Y], Z) = 0, \forall Z \in \g \} \\
           &=& \{ X \in \g : B(Y, [X,Z])=0, \forall Z \in \g \}.
\end{array}           
\] 
The bilinear form $B$ induces a non-degenerate alternating form $B_{Y} : \g/Y^{\#} \times \g/Y^{\#} \longrightarrow E$ defined by $B_{Y}(X_{1}, X_{2})=B(Y, [X_{1}, X_{2}])$. \\

Let $L \subset \mathfrak{g}$ be a lattice satisfying the following conditions:
\begin{enumerate}
\item $[L, L] \subset L$,
\item $L = \oplus_{i \in \Z} L_{i}$, where $L_{i}=L \cap \g_{i}$,
\item The lattice $L/L_{Y}$, where $L_{Y} = L \cap Y^{\#}$, is self dual (i.e. $(L/L_{Y})^{\perp} = L/L_{Y}$) with respect to $B_{Y}$. (For any vector space $V$ with a non-degenerate bilinear form $B$ and a lattice $M$ in $V$, $M^{\perp} := \{ X \in V : B(X, Y) \in \mathfrak{O}_{E}, \forall Y \in V \}$.)
\end{enumerate}
A lattice $L$ satisfying the above properties can be chosen by taking a suitable basis of all $\g_{i}$'s, see \cite{MW87}. Now we summarize a few well known properties of the exponential map, and use them to define subgroups $G_{n}$ and their Iwahori decompositions.
\begin{lemma} \label{subgroup G_n}
\begin{enumerate}
\item There exists a positive integer $A$ such that $\exp$ is defined and injective on $\varpi^{A}L$, with inverse $\log$. 
\item The $\exp$ map on $\varpi^{n}L$ is homeomorphic onto its image image $G_{n}:=\exp(\varpi^{n}L)$, which is an open subgroup of $G$ for all $n \geq A$. 
\item Set $P_{n}^{-} = \exp(\varpi^{n}L \cap \mathfrak{p}^{-})$ and $N_{n} = \exp(\varpi^{n}L \cap \mathfrak{n})$. Then we have an Iwahori factorization
\[
G_{n} = P_{n}^{-}N_{n}.
\]
\end{enumerate}
\end{lemma}
We will be working with a certain character $\chi_{n}$ of $G_{n}$, which we recall in the next lemma. 
\begin{lemma} \label{character_chi_n}
For large values of $n$ there exists a character $\chi_{n}$ of $G_{n}$, whose restriction to $\exp((Y^{\#} \cap \varpi^{n}L)+\varpi^{n + \val 2}L)$ coincides with $\gamma \mapsto \psi(B(\varpi^{-2n}Y, \log \gamma))$. If $P_{n}^{-}$ is as in Lemma \ref{subgroup G_n}, the character $\chi_{n}$ can be chosen so that 
\[
\chi_{n}(p)=1, \forall p \in P_{n}^{-}.
\]
\end{lemma}
For a proof of this lemma and other properties of this character $\chi_{n}$ see Lemma 5 in \cite{San14}.

\begin{remark} \normalfont
If $p \neq 2$, then the map $\gamma \mapsto \psi(B(\varpi^{-2n}Y, \log \gamma))$ itself defines a character of $G_{n}$ for large $n$ and satisfies the properties stated in Lemma \ref{character_chi_n}. But for $p=2$, there are more that one characters $\chi_{n}$, for more details see \cite{San14}.
\end{remark}

\section{Covering groups} \label{covering_groups}
Let $\mu_{r}$ be the group of $r$-th roots of unity in $\C$. Consider an $r$-fold covering $\tilde{G}$ of $G$, which is a central extension of locally compact groups of the group $G$ by $\mu_{r}$ giving rise to the following short exact sequence:
\[
1 \longrightarrow \mu_{r} \longrightarrow \tilde{G} \longrightarrow G \longrightarrow 1.
\]
\begin{lemma} \label{splitting of covering}
\begin{enumerate}
\item The covering $\tilde{G} \longrightarrow G$ splits uniquely over any unipotent subgroup of $G$.
\item For large enough $n$ the covering $\tilde{G} \longrightarrow G$ splits over $G_{n}$. Moreover, there is a splitting $s$ of $\tilde{G} \longrightarrow G$ restricted to $\cup_{g \in G} g G_{n} g^{-1}$ such that $s(hth^{-1}) = h s(t) h^{-1}$ for all $h \in G$.
\end{enumerate}
\end{lemma}
\begin{proof}
\begin{enumerate}
\item This is well known, see \cite{MW95}. For a simpler proof, in the case when $E$ has characteristic zero, see Section 2.2 of \cite{WWLi14}.
\item Recall that the subgroups $G_{n}$ form a basis of neighbourhoods of identity. It is well known that the covering $\tilde{G} \longrightarrow G$ splits over a neighbourhood of identity. Therefore for large enough $n$, the covering splits over $G_{n}$. There are more than on possible splitting of the cover $\tilde{G} \longrightarrow G$ over $G_{n}$.  If a splitting is fixed, then any other splitting over $G_{n}$ will differ from the above splitting by a character $G_{n} \longrightarrow \mu_{r}$. \\
Fix some $m$ such that the covering splits over $G_{m} = \exp(\varpi^{m}L)$. As mentioned above, any two splitting over the subgroup $G_{m}$ will differ by a character $G_{m} \longrightarrow \mu_{r}$ and any such character is trivial over
\[
G_{m}^{r} := \{ g^{r} : g \in G_{m} \}.
\] 
Hence all the possible splittings over $G_{m}$ agree on $G_{m}^{r}$. The subset $G_{m}^{r}$ is a subgroup of $G_{m}$ as it equals $\exp(r \cdot \varpi^{m} L)$.   Let $g, h \in G$. We have
\[
(g G_{m} g^{-1} \cap h G_{m}h^{-1}) \supset (g G_{m}^{r} g^{-1} \cap h G_{m}^{r} h^{-1}).
\]
This implies that any two splittings of $\tilde{G} \longrightarrow G$ restricted to $g G_{m}^{r} g^{-1} \cap h G_{m}^{r} h^{-1}$ one coming from the restriction of a splitting of $\tilde{G} \longrightarrow G$ over $g G_{m} g^{-1}$ and the other coming from the restriction of a splitting over $h G_{m} h^{-1}$ are the same.. Now choose $A'$ so large such that $G_{n} \subset G_{m}^{r}$ for $n \geq A'$. We fix the  splitting of $G_{n}$ which comes from that of the restriction of $G_{m}^{r}$. This gives us a splitting over $\cup_{g \in G} g G_{n} g^{-1}$. \qedhere
\end{enumerate}
\end{proof}
Using this splitting we get that exponential map is defined from a small enough neighbourhood of $\g$ to $\tilde{G}$, namely the usual exponential map composed with this splitting, which one can use to define the character expansion of an irreducible admissible genuine representation $(\pi, W)$ of $\tilde{G}$, which has been done by Wen-Wei Li in \cite{WWLi}.
\begin{remark} \normalfont
If $r$ is co-prime to $p$, then as $G_{n}$ is a pro-$p$ group and $(r,p)=1$, there is no non-trivial character from $G_{n}$ to $\mu_{r}$. In that situation, the splitting in the above lemma is unique.
\end{remark}
From now onwards, for large enough $n$, we treat $G_{n}$ not only as a subgroup of $G$ but also as one of $\tilde{G}$, with the above specified splitting. In other words, for the covering group $\tilde{G}$ (as in the linear case) we have a sequence of pairs $(G_{n}, \chi_{n})$ using the splitting specified above which satisfies the properties described in Section 2.
\begin{definition} \label{phi tilde}
\normalfont
 Let $H \subset G$ be an open subgroup and $s : H \hookrightarrow \tilde{G}$ be a splitting. Then for any $\phi \in C_{c}^{\infty}(G)$ with ${\rm supp}(\phi) \subset H$ define $\tilde{\phi}_{s} \in C_{c}^{\infty}(\tilde{G})$  as follows:
\[
 \tilde{\phi}_{s}(g) := \left\{ \begin{array}{ll}
                           \phi(g'), & \text{ if } g=s(g') \in s(H) \\
                           0, & \text{ if } g \in \tilde{G} \backslash s(H)
                           \end{array}
                          \right.
\]
\end{definition}
Note that this definition depends upon the choice of splitting. Whenever the splitting is clear in the context or it has been fixed and there is no confusion we write just $\tilde{\phi}$ instead of $\tilde{\phi}_{s}$ and $H$ for $s(H)$. Recall that the convolution $\phi \ast \phi'$ for $\phi, \phi' \in C_{c}^{\infty}(G)$ is defined by
\[
 \phi \ast \phi' (x) = \int_{G} \phi(xy^{-1}) \phi'(y) \, dy.
\]
Observe that 
\[
\supp(\phi \ast \phi') \subset \supp(\phi) \cdot \supp(\phi'),
\]
which implies the lemma below.
\begin{lemma} \label{convolution}
Let $H$ be an open subgroup of $G$ such that the covering $\tilde{G} \rightarrow G$ has a splitting over $H$, say, $s : H \hookrightarrow \tilde{G}$, satisfying $s(xy)=s(x)s(y)$ whenever $x, y$ are in $H$. If $\phi, \phi' \in C_{c}^{\infty}(G)$ are such that supp($\phi$) and supp($\phi'$) are contained in $H$, then we have 
\[
 \widetilde{\phi \ast \phi'} =\tilde{\phi} \ast \tilde{\phi'}.
\]
\end{lemma}

\section{Degenerate Whittaker forms} \label{degenerate_W_forms}
In this section we give the definition of degenerate Whittaker forms for a smooth genuine representation $\pi$ of $\tilde{G}$. This is an adaptation of Section I.7 of \cite{MW87} and Section 5 of \cite{San14}. \\

Define $N := \exp(\mathfrak{n}) = \exp(\oplus_{i \geq 1} \g_{i}))$, $N^{2} := \exp(\oplus_{i \geq 2} \g_{i})$ and $N' := \exp(\mathfrak{g}_{1} \cap Y^{\#})N^{2}$. It is easy to see that $N^{2}$ and $N'$ are normal subgroups of $N$. Let $H$ be the Heisenberg group defined with $\mathfrak{g}_{1}/(\mathfrak{g}_{1} \cap Y^{\#}) \times E$ as underlying set using the symplectic form induced by $B_{Y}$, i.e. for $X, Z \in \g_{1}/(\g_{1} \cap Y^{\#})$ and $a, b \in E$,
\begin{equation} \label{definition: H}
(X, a)(Z, b) = (X+Z, a+b+\frac{1}{2}B_{Y}(X,Z)).
\end{equation}  
Consider the map $N \longrightarrow H$ given by
\[
\exp(X) \mapsto (\bar{X}, B(Y, X)),
\]
where $\bar{X}$ is the image of the $\mathfrak{g}_{1}$ component of $X$ in $\mathfrak{g}_{1}/(\mathfrak{g}_{1} \cap Y^{\#})$. The Campbell-Hausdorff formula implies that the above map is a homomorphism with the following kernel 
\[
N'' = \{ n \in N' : B(Y, \log n)=0 \}.
\] 
Let $\chi : N' \longrightarrow \C^{\times}$ be defined by $\gamma \mapsto \psi \circ B(Y, \log \gamma)$. Note that $\gamma \mapsto B(Y, \log \gamma) \in E \cong \{0\} \times E \subset H$ induces an isomorphism $N'/N'' \cong E$.	\\

We note that the cover $\tilde{G} \longrightarrow G$ splits uniquely over the subgroups $N, N'$ and $N''$. We denote the images of these splittings inside $\tilde{G}$ by the same letters. For a smooth genuine representation $(\pi, W)$ of $\tilde{G}$ we define
\[
N^{2}_{\chi}W= \{ \pi(n)w - \chi(n)w : w \in W, n \in N^{2} \}
\]
and 
\[
N_{\chi}'W= \{ \pi(n)w - \chi(n)w : w \in W, n \in N' \}.
\]
Note that $N$ normalizes $\chi$, therefore $H=N/N''$ acts on $W/N_{\chi}'W$ in a natural way. This action restricts to $N'/N''$ ( the center of $N/N''$) as multiplication by the character $\chi$. Let $\mathcal{S}$ be the unique irreducible representation of the Heisenberg group $H$ with central character $\chi$. 
\begin{definition} \normalfont
Define the space of degenerate Whittaker forms for $(\pi, W)$ associated to $(Y, \varphi)$ to be 
\[
\mathcal{W} := \Hom_{H}( \mathcal{S}, W/N_{\chi}'W).
\]
\end{definition}
\begin{remark} \normalfont
If $\g_{1} =0$, then $N = N' = N^{2}$. In this case, $\mathcal{W} \cong W/N_{\chi}W$ is the $(N, \chi)$-twisted Jacquet functor.  
\end{remark}
\begin{definition} \normalfont
For a smooth representation $(\pi,W)$ of $\tilde{G}$ define $\mathcal{N}_{\Wh}(\pi)$ to be the set of nilpotent orbits $\mathcal{O}$ of $\mathfrak{g}$ such that there exists $Y \in \mathcal{O}$ and $\varphi$ as in Equation \ref{property:1}, such that the space of degenerate Whittaker forms for $\pi$ associated to $(Y, \varphi)$ is non-zero.
\end{definition}
As $\g_{1}/\g_{1} \cap Y^{\#}$ is a symplectic vector space and $L/L_{Y}$ is self dual, it follows that $L_{H} := (L \cap \g_{1})/(L \cap \g_{1} \cap Y^{\#})$ is a self dual lattice in the symplectic vector space $H/Z(H) \cong \g_{1}/(\g_{1} \cap Y^{\#})$. \\

Recall the definition of the Heisenberg group $H$ (see Equation \ref{definition: H}) and as $\psi$ is trivial on $\mathfrak{O}_{E}$, it follows that one can extend the character $\psi$ of $E \cong Z(H)$ to a character of the inverse image of $2L_{H}$ under $H \longrightarrow \g_{1} / (\g_{1} \cap Y^{\#})$ by defining it to be trivial on $2L_{H} \times \{0\} \subset H$. From Lemma 4 in \cite{San14}, this character can be extended to a character $\tilde{\chi}$ on the inverse image $H_{0}$ of $L_{H}$ under the natural map $H \longrightarrow \g_{1}/(\g_{1} \cap Y^{\#})$.\\

\begin{remark} \normalfont
There are one parameter subgroups $\varphi$ which do not arise from $\mathfrak{sl}_{2}$-triplets. If $\varphi$ arises from  $\mathfrak{sl}_{2}$-triplets, then it is easy to see that $Y^{\#} \subset \oplus_{i \leq 0} \g_{i}$. In particular we have $\g_{1} \cap Y^{\#} = \{ 0 \}$ and hence the Heisenberg group $H = \g_{1} \times E$.
\end{remark}

Then, by Chapter 2, Section I.3 of \cite{MVW}, one knows that $\mathcal{S} = {\rm ind}_{H_{0}}^{H} \tilde{\chi}$, induction with compact support. Since $H_{0}$ is an open subgroup of the locally profinite group $H$, we have the following form of the Frobenius reciprocity law:
\[
\Hom_{H}(\mathcal{S}, \tau) = \Hom_{H}({\rm ind}_{H_{0}}^{H} \tilde{\chi}, \tau) = \Hom_{H_{0}}(\tilde{\chi}, \tau\mid _{H_{0}})
\] 
for any smooth representation $\tau$ of $H$. Thus, in the category of representations of $N$ on which $N'$ acts via the character $\chi$, the functor $\Hom_{H}(\mathcal{S}, -)$ amounts to taking the $\tilde{\chi}\mid _{H_{0}}$-isotypic component. Since $H_{0}$ is compact modulo the center, this functor is exact. Thus we have
\[
\mathcal{W} = \Hom_{H}(\mathcal{S}, W/N_{\chi}'W) \cong (W/N_{\chi}'W)^{(H_{0}, \tilde{\chi})}.
\]
where $(W/N_{\chi}'W)^{(H_{0}, \tilde{\chi})}$ denotes the $(H_{0}, \tilde{\chi})$-isotypic component of $W/N_{\chi}'W$. \\

Recall that we have defined certain characters $\chi_{n}$'s in Section \ref{G_n ans chi_n} and now we have a character $\tilde{\chi}$. We need to choose them in a compatible way. First we fix a character $\tilde{\chi}$ and consider it as a character of $\exp(\g_{1} \cap L)N'$ in the obvious way (as $\exp(\g_{1} \cap L)N'$ is the inverse image of $H_{0}$ under $N \longrightarrow H$). Let $t:=\varphi(\varpi) \in G$. Let $\tilde{t} \in \tilde{G}$ be any lift of $t$ in $\tilde{G}$. Let 
\[
G_{n}' = \Int({\tilde{t}}^{-n})(G_{n}), P_{n}'= \Int({\tilde{t}}^{-n})(P_{n}^{-}) \, {\rm and} \, V_{n}'= \Int({\tilde{t}}^{-n})(N_{n}).
\] 
It can be easily verified that $V_{n}'$ contains $\exp(\g_{1} \cap L)$. We also have $V_{n}' \subset V_{m}'$ for large $m, n$ with $n \leq m$. Moreover
\[
\exp(\g_{1} \cap L)N^{2} = \bigcup_{n \geq 0} V_{n}'.
\] 
It can also be verified easily that $\tilde{\chi} \circ \Int({\tilde{t}}^{-n})$ restricts to a character of $N_{n}$ that extends the character on $N_{n+ \val 2}N_{n}'$ given by $\gamma \mapsto \psi(B(\varpi^{-2n}Y, \log \gamma))$. Now define 
\begin{equation} \label{chi_n is character}
\chi_{n}(pv) = \tilde{\chi}(\tilde{t}^{-n}v \tilde{t}^{n}), \forall p \in P_{n}^{-} \text{ and } \forall v \in V_{n}'.
\end{equation}
\begin{lemma} [Lemma 6 in \cite{San14}]
Let $\chi_{n}$ be as defined in Equation \ref{chi_n is character}. Then $\chi_{n}$ is a character of $G_{n}$ and satisfies the properties stated in Lemma \ref{character_chi_n}. 
\end{lemma}
Define a character $\chi_{n}'$ on $G_{n}'$ as follows:
\[
\chi_{n}' := \chi_{n} \circ \Int({\tilde{t}}^{n}).
\]
\begin{remark} \normalfont
 The characters $\chi_{n}$ have been so defined that $\chi_{n}'$ agree with $\chi$ on the intersection of their domains, namely, for large $n$ we have,
\[
\chi_{n}'\mid _{V_{n}'} = \tilde{\chi}\mid _{V_{n}'}.
\]
In particular, $\chi_{n}'\mid _{\exp(L \cap \g_{1})} = \tilde{\chi}\mid _{\exp(L \cap \g_{1})}$. One can also see that $\chi_{n}'$ and $\chi_{m}'$ (for large $n,m$) agree on $G_{n}' \cap G_{m}'$, because they agree on $V_{n}' \cap V_{m}'$ and also on $P_{n}' \cap P_{m}'$ (being trivial on it).
\end{remark}
Set
\begin{equation}
W_{n} := \{ w \in W \mid  \pi(\gamma)w = \chi_{n}(\gamma)w, \forall \gamma \in G_{n} \}
\end{equation}
and 
\begin{equation}
W_{n}' := \{ w \in W \mid  \pi(\gamma)w = \chi_{n}'(\gamma)w, \forall \gamma \in G_{n}' \} = \pi({\tilde{t}}^{-n})W_{n}
\end{equation}
For large $m, n$ define the map $I_{n,m}' : W_{n}' \longrightarrow W_{m}'$ by
\begin{equation}
I_{n,m}'(w) = \int_{G_{m}'} \chi_{m}'(\gamma^{-1}) \pi(\gamma) w \, d\gamma.
\end{equation}
Let $m, n$ be large with $m > n$. Since $\chi_{n}'$ is trivial on $P_{n}' \supset P_{m}'$ and since $G_{m}' = P_{m}'V_{m}'$ and for a convenient choice of measures we have
\[
\begin{array}{lll}
I_{n,m}'(w) &=& \int_{V_{m}'} \chi_{m}'(x^{-1}) \pi(x) w \, dx \\
				 &=& \int_{\exp(\g_{1} \cap L)} \tilde{\chi}^{-1}(\exp X) \pi(\exp X) \int_{N^{2} \cap G_{m}'} \chi (x^{-1}) \pi(x) w \, dx \, dX.
\end{array}
\]
Now using the fact that $\exp(\g_{1} \cap L)$ lies in $G_{n}'$ for large $n$ and that it normalizes the character $\chi|_{N^{2}}$, we get
\[
\begin{array}{lll}
I_{n,m}'(w) &=& \int_{N^{2} \cap G_{m}'} \chi(x^{-1}) \pi(x) w \, dx \\
				 &=& \int_{N' \cap G_{m}'} \chi(x^{-1}) \pi(x) w \, dx.
\end{array}
\]
From this the following is clear for large $n, m$ with $m > n$
\begin{equation} \label{composition of I'}
I_{n,m}' = I_{n+1,m}' \circ I_{n,n+1}'.
\end{equation}
For large $n$, the above equation gives that $\ker I_{n,m}' \subset \ker I_{n,p}'$ for $n < m \leq p$. Set $W_{n, \chi}':= \cup_{m>n} \ker I_{n,m}'$. Recall that for any unipotent subgroup $U$, a character $\chi : U \longrightarrow \C^{\times}$ and $w \in W$;  $\int_{K} \chi(x)^{-1} \pi(x) w \, dx = 0$ for some open compact subgroup $K$ of $U$ if and only if $w \in U_{\chi}W$, where $U_{\chi}W$ is the span of $\{ \pi(u)w - \chi(u)w \mid u \in U, w \in W \}$. Thus we have $W_{n, \chi} \subset N^{2}_{\chi}W$ as well as $W_{n, \chi} \subset N'_{\chi}W$, which gives the following natural maps 
\[
j_{n} : W_{n}'/W_{n,\chi}' \longrightarrow W/N_{\chi}^{2}W \text{ and } j_{n}' : W_{n}'/W_{n,\chi}' \longrightarrow W/N_{\chi}'W
\] 
and these give the following diagram:
\begin{equation} \label{commtative diagram}
\xymatrix{ W_{n}'/W_{n, \chi}' \ar[rr]^{j_{n}'} \ar[dr]^{j_{n}} & & W/N_{\chi}'W \\ & W/N_{\chi}^{2}W \ar@{-->}[ru]_{\exists \, {\rm natural}} & }
\end{equation}
By the compatibility between $\chi_{n}'$ and $\tilde{\chi}$, it is easy to see that the image of $j_{n}'$ is contained in $(W/N_{\chi}'W)^{(H_{0}, \tilde{\chi})}$. Let $w \in W$ such that the image $\bar{w}$ of $w$ in $W/N_{\chi}'W$ belongs to $(W/N_{\chi}'W)^{(H_{0}, \tilde{\chi})}$. For large $n$, $P_{n}'$ acts trivially on $w$, as $(\pi, W)$ is smooth. Since $G_{n}' = P_{n}'V_{n}'= V_{n}'P_{n}'$, the element
\[
\int_{V_{n}'} \chi_{n}'(x^{-1}) \pi(x) w \, dx
\]
belongs to $W_{n}'$. As $\chi_{n}'$ and $\chi$ are compatible, it can be seen that its image in $W/N_{\chi}'W$ is $\bar{w}$. This gives us the following lemma.

\begin{lemma} \label{W_n non zero}
Let $(Y, \varphi)$ be arbitrary. Then any element of $(W/N_{\chi}'W)^{(H_{0}, \chi)}$ belongs to $j_{n}'(W_{n}')$ for all sufficiently large $n$. In particular, if $\mathcal{W} \neq 0$ then, for large $n$, $W_{n}$ and $W_{n}'$ are non-zero.
\end{lemma}

\section{Main theorem} \label{proof of main theorem}
Now recall that, by the work of Wen-Wei Li \cite{WWLi}, the Harish-Chandra-Howe character expansion of an irreducible admissible genuine representation of $\tilde{G}$ at the identity element has an expression of the same form as that of an irreducible admissible representation of a linear group. The proof of the following lemma for a covering group follows verbatim that of Proposition I.11 in \cite{MW87} and Proposition 1 in \cite{San14}.
\begin{proposition} \label{W is non-zero} 
Let $\mathcal{W}$ be the space of degenerate Whittaker forms for $\pi$ with respect to a given $(Y, \varphi)$. If $\mathcal{W} \neq 0$ then there exists a nilpotent orbit $\mathcal{O}$ in $\mathcal{N}_{\tr}(\pi)$ such that $\mathcal{O}_{Y} \leq \mathcal{O}$ (i.e.,  $Y \in \bar{\mathcal{O}}$).
\end{proposition}
Let the function $\phi_{n} : G \longrightarrow \C$ be defined by
\[
 \phi_{n}(\gamma) = \left\{ \begin{array}{ll}
 											\chi_{n}(\gamma^{-1}), & \text{ if } \gamma \in G_{n} \\
 											0, & \text{ otherwise. }
 											\end{array} \right.
\]
Consider the corresponding function $\tilde{\phi_{n}} : \tilde{G} \longrightarrow \C$.
Write the character expansion at the identity element as follows:
\[
\Theta_{\pi} \circ \exp= \sum_{\mathcal{O}} c_{\mathcal{O}} \widehat{\mu_{\mathcal{O}}}.
\]
Choose $n$ large enough so that the above expansion is valid over $G_{n}$ and then evaluate $\Theta_{\pi}$ at the function $\tilde{\phi_{n}}$. As $\pi(\tilde{\phi_{n}})$ is a projection from $W$ to $W_{n}$, by definition we get $\Theta_{\pi}(\tilde{\phi_{n}}) = \trace \, \pi(\tilde{\phi_{n}}) = \dim W_{n}$. Now assume that $(Y, \varphi)$ is such that $O_{Y}$ is a maximal element in $\mathcal{N}_{tr}(\pi)$. On the other hand, if we evaluate $\sum_{\mathcal{O}} c_{\mathcal{O}} \widehat{\mu_{\mathcal{O}}}(\tilde{\phi_{n}})$, it turns out that $\widehat{\mu_{\mathcal{O}}}(\tilde{\phi_{n}})$ is zero unless $\mathcal{O} = \mathcal{O}_{Y}$. In addition, if we fix a $G$-invariant measure on $\mathcal{O}_{Y}$ as in I.8 of \cite{MW87} (for more details about this invariant measure see Section 3 of \cite{San14}), we get the following lemma.
\begin{lemma} \label{dimW_n is c_O} [Lemma I.12 in \cite{MW87} and Lemma 7 in \cite{San14}] \\
If $(Y, \varphi)$ is such that $\mathcal{O}_{Y}$ is a maximal element of $\mathcal{N}_{\tr}(\pi)$. Then for large $n$, 
\[
\dim W_{n} = c_{\mathcal{O}_{Y}}.
\]
In particular, the dimension of $W_{n}$ is finite and independent of $n$, for large $n$.
\end{lemma}
From Lemma \ref{W_n non zero} we know that every vector in $\mathcal{W}$ is in the image of $j_{n}'$ for large $n$. In particular, if $W_{n}$ is finite dimensional, we get that the map $j_{n}'$ is surjective. Moreover, we have the following lemma whose proof is verbatim that of Corollary I.14 in \cite{MW87} and Lemma 8 in \cite{San14} in the case of a linear group.
\begin{lemma} \label{j_n and j_n'} 
Let $(Y, \varphi)$ is such that $\mathcal{O}_{Y}$ is a maximal element of $\mathcal{N}_{\tr}(\pi)$. Then for large $n$, the maps $j_{n}$ and $j_{n}'$ are injections and the image of $j_{n}'$ is $(W/N_{\chi}'W)^{(H_{0}, \tilde{\chi})}$.
\end{lemma}

Let $\phi_{n}' : G \longrightarrow \C$ be defined by
\[
 \phi_{n}'(\gamma) = \left\{ \begin{array}{ll}
 											\chi_{n}'(\gamma^{-1}), & \text{ if } \gamma \in G_{n}' \\
 											0, & \text{ otherwise. }
 											\end{array} \right.
\]

Consider the corresponding function $\tilde{\phi}_{n}' : \tilde{G} \longrightarrow \C$.
\begin{lemma} \label{injectivity}
Consider a pair $(Y, \varphi)$ such that $\mathcal{O} = \mathcal{O}_{Y}$ is a maximal in $\mathcal{N}_{\tr}(\pi)$. Then for large enough $n$:
\begin{enumerate} 
\item Let $\mathcal{Y}_{n} \subset  G_{n+1}' \cap G(Y)$ be a set of representatives for the $G_{n}'$ double cosets in $G_{n}'(G_{n+1} \cap G(Y))G_{n}'$. Then large enough $n$, 
\[
\tilde{\phi}_{n}' \ast \tilde{\phi}_{n+1}' \ast \tilde{\phi}_{n}'(g)= \left\{ \begin{array}{ll}
																					\lambda \cdot (\chi_{n}')^{-1}(h_{1}h_{2}), & \text{ if } g=h_{1}yh_{2} \text{ with } y \in \mathcal{Y}_{n},  h_{1}, h_{2} \in G_{n}' \\
																					0, & \text{ if } g \notin G_{n}'\mathcal{Y}_{n}G_{n}'
																					\end{array} \right.
\]
where $\lambda = \meas(G_{n}' \cap G_{n+1}') \meas(G_{n}')$.
\item For large $n$, $I_{n, n+1}'$ is injective.
\end{enumerate}
\end{lemma}
\begin{proof}
From part (a) of Lemma 9 in \cite{San14}, we have
\[
\phi_{n}' \ast \phi_{n+1}' \ast \phi_{n}' (g)= \left\{ \begin{array}{ll}
																					\lambda \cdot (\chi_{n}')^{-1}(h_{1}h_{2}), & \text{ if } g=h_{1}yh_{2} \text{ with } y \in \mathcal{Y}_{n},  h_{1}, h_{2} \in G_{n}' \\
																					0, & \text{ if } g \notin G_{n}'\mathcal{Y}_{n}G_{n}'
																					\end{array} \right.
\]
where $\lambda = \meas(G_{n}' \cap G_{n+1}') \meas(G_{n}')$. Now part 1 follows from Lemma \ref{convolution}, as we have 
\begin{equation} 
\tilde{\phi}_{n}' \ast \tilde{\phi}_{n+1}' \ast \tilde{\phi}_{n}' = \oversortoftilde{(\phi_{n}' \ast \phi_{n+1}' \ast \phi_{n}')}.
\end{equation}
Now we prove part 2. It is enough to say that $\pi( \tilde{\phi}_{n}' \ast \tilde{\phi}_{n+1}' \ast \tilde{\phi}_{n}')$ acts by a non-zero multiple of identity on $W_{n}'$. This implies that $I_{n+1,n}' \circ I_{n,n+1}'$ is a non-zero multiple of identity on $W_{n}'$. From part 1 we get that $\tilde{\phi}_{n}' \ast \tilde{\phi}_{n+1}' \ast \tilde{\phi_{n}'}$ is a positive linear combination of functions $\tilde{\phi}_{n,y}' : \gamma \mapsto \tilde{\phi}_{n}'(\gamma y^{-1})$, where $y \in G_{n+1} \cap G(Y)$ is fixed and $G(Y)$ is centralizer of $Y$ in $G$. Then the lemma follows from the fact that $\pi(y)$ acts trivially on $W_{n}'$ for large $n$, so that
\[
\pi(\tilde{\phi}_{n,y}')|_{W_{n}'} = \pi(\tilde{\phi}_{n}') \pi(y)|_{W_{n}'} = \pi(\tilde{\phi}_{n}')|_{W_{n}'}. \qedhere
\]
\end{proof}

\begin{theorem}
Let $(\pi, W)$ be an irreducible admissible genuine representation of $\tilde{G}$. 
\begin{enumerate}
\item  The set of maximal elements in $\mathcal{N}_{\tr}(\pi)$ coincides with the set of maximal elements in $\mathcal{N}_{Wh}(\pi)$.
\item Let $\mathcal{O}$ be a maximal element in $\mathcal{N}_{\tr}(\pi)$. Then the coefficient $c_{\mathcal{O}}$ equals the dimension of the space of degenerate Whittaker forms with respect to any pair $(Y, \varphi)$ such that with $Y \in \mathcal{O}$ is arbitrary and $\varphi : \mathbb{G}_{m} \longrightarrow {\bf G}$ satisfies $\Ad(\varphi(s))Y = s^{-2}Y$ for all $s \in E^{\times}$.
\end{enumerate}
\end{theorem}
\begin{proof}
Let $\mathcal{O}$ be a maximal element in $\mathcal{N}_{\tr}(\pi)$. Choose $(Y, \varphi)$ such that $Y \in \mathcal{O}$ and $\varphi : \mathbb{G}_{m} \longrightarrow {\bf G}$ satisfying $\Ad(\varphi(s))Y = s^{-2}Y$. Then, from Lemma \ref{dimW_n is c_O}, for large $n$ we have
\[
\dim W_{n} = c_{\mathcal{O}}.
\]
Therefore $W_{n} \neq 0$ (resp $W_{n}' \neq 0$) for large $n$ . From Lemma \ref{j_n and j_n'}, the map $j_{n}'$ is injective and maps surjectively onto $(W/N_{\chi}'W)^{(H_{0}, \tilde{\chi})}$. But from second part of Lemma \ref{injectivity} and Equation \ref{composition of I'}, $I_{n,m}'$ is injective for large $n$ and $m>n$ which implies that $W_{n, \chi}' = \cup_{m>n} \ker(I_{n,m}') =0$. Thus $\dim \mathcal{W} = \dim W_{n}' = \dim W_{n} = c_{\mathcal{O}}$, which proves part 2 of the theorem.  In particular, $\mathcal{W} \neq 0$ and hence $\mathcal{O} \in  \mathcal{N}_{\Wh}(\pi)$. Now we claim that $\mathcal{O}$ is maximal in $\mathcal{N}_{\Wh}(\pi)$. If not, there is a maximal orbit $\mathcal{O}' \in \mathcal{N}_{\Wh}(\pi)$ such that $\mathcal{O} \lneq \mathcal{O}'$. From Proposition \ref{W is non-zero}, there is $\mathcal{O}'' \in \mathcal{N}_{\tr}(\pi)$ such that $\mathcal{O}' \leq \mathcal{O}''$. Therefore $\mathcal{O} \lneq \mathcal{O}''$ and $\mathcal{O}, \mathcal{O}'' \in \mathcal{N}_{\tr}(\pi)$, a contradiction to the maximality of $\mathcal{O}$ in $\mathcal{N}_{\tr}(\pi)$. \\

Let $\mathcal{O}$ be a maximal element in $\mathcal{N}_{\Wh}(\pi)$. From Proposition \ref{W is non-zero}, there is an element in $\mathcal{O}' \in \mathcal{N}_{\tr}(\pi)$ such that $\mathcal{O} \leq \mathcal{O}'$. Take such a maximal $\mathcal{O}'$. Then by the result in the above paragraph, $\mathcal{O}'$ is a maximal element in $\mathcal{N}_{\Wh}(\pi)$. But $\mathcal{O}$ is also maximal in $\mathcal{N}_{\Wh}(\pi)$. Hence $\mathcal{O}=\mathcal{O}'$. This proves that $\mathcal{O}$ is a maximal element in $\mathcal{N}_{\tr}(\pi)$ too. 
\end{proof}


\begin{thebibliography}{99}
\bibitem{WWLi}
  W.-W. Li:
  \emph{La formule des traces pour les rev\^{e}tements de groupes r\'{e}ductifs connexes. II. Analyse harmonique locale}
  Ann. Scient. \'{E}c. Norm. Sup. 4 e s\'{e}rie, {\bf 45}, 787-859, (2012).
  
\bibitem{WWLi14}
  W.-W. Li:
  \emph{La formule des traces pour les revêtements de groupes r\'{e}ductifs connexes. I. Le d\'{e}veloppement g\'{e}om\'{e}trique fin},                    
  Journal f\"{u}r die reine und angewandte Mathematik. {\bf 686}, 37-109, (2014).

\bibitem{MVW}
  C. M\oe glin, M. F. Vigneras and J.-L. Waldspurger:
  \emph{Correspondances de Howe sur un corps $p$-adique},
  Lecture notes in Mathematics,  Springer-Verlag, Berlin, {\bf 1291}, (1987).

\bibitem{MW87}
  C. M\oe glin, and J.-L. Waldspurger:
  \emph{Mod\`eles de Whittaker d\'eg\'en\'er\'es pour de groupes $p$-adiques},
  Mathematics Zeitschrift, {\bf 196},  427-452, (1987).

\bibitem{MW95}
  C. M\oe glin and J.-L. Waldspurger:
  \emph{Spectral decomposition and Eisenstein series},
  Appendix I, Lifting of Unipotent subgroups into a central extension, Cambridge tracts in Mathematics, Cambridge University press, Cambridge, {\bf 113}, 273-277, (1995).

\bibitem{Prasad92}
  D. Prasad:
  \emph{Invariant forms for the representations of $GL_2$ over a local field}.
  American Journal of Mathematics, {\bf 114}, 1317-1363, (1992).


\bibitem{Rod75}
 F. Rodier:
  \emph{Mod\`eles de Whittaker et caract\'eres de repr\'{e}sentations},
  Non commutative harmonic analysis (Actes Colloq., Marseille-Luminy, 1974), Lecture Notes in Mathematics, Springer, Berlin, {\bf 466} 151-171, (1975).

\bibitem{San14}
Sandeep Varma:
  \emph{On a result of M\oe glin and Waldspurger in residual characteristic 2},
  To appear in Mathematische Zeitschrift.

\end{thebibliography}
\end{document}